\documentclass[12pt]{amsart}

\usepackage{hyperref}
\usepackage{setspace}
\usepackage{latexsym}
\usepackage{amsmath}
\usepackage{amsfonts}
\usepackage{amssymb}
\usepackage{amsthm}
\usepackage{verbatim}

\usepackage[left=3cm,top=2cm,right=3cm,bottom = 2cm]{geometry}

\newcommand{\hstar}{\bb{H}^\ast}

\newcommand{\uptrithree}[6]{\left(\begin{array}{ccc}#1&#2&#3\\0&#4&#5\\0&0&#6\end{array}\right)}
\newcommand{\threemat}[9]{\left(\begin{array}{ccc}#1&#2&#3\\#4&#5&#6\\#7&#8&#9\end{array}\right)}
\newcommand{\cstar}{\bb{C}^\times}

\newcommand{\diag}[1]{\mbox{diag}\left\{#1\right\}}

\newcommand{\eqr}[1]{\mbox{(\ref{eq:#1})}}

\newcommand{\e}[1]{\textbf{e}\left(\textstyle#1\right)}

\newcommand{\G}{\Gamma }

\newcommand{\mat}[4]{\left(\begin{array}{rr}#1&#2\\#3&#4\end{array}\right)}

\newcommand{\bb}[1]{\mathbb{#1}}

\newcommand{\gln}[1]{\GL_{#1}(\bb{C})}

\newtheorem{thm}{Theorem}[section]

\newtheorem{cor}[thm]{Corollary}
\newtheorem{prop}[thm]{Proposition}

\DeclareMathOperator{\SL}{SL}
\newcommand{\ZZ}{\mathbb{Z}}
\DeclareMathOperator{\GL}{GL}
\newcommand{\CC}{\mathbb{C}}
\usepackage{color}
\newcommand{\cV}{\mathcal{V}}

\newcommand{\cMbar}{\overline{\cM}}
\newcommand{\cM}{\mathcal{M}}
\newcommand{\cL}{\mathcal{L}}

\DeclareMathOperator{\PSL}{PSL}
\DeclareMathOperator{\Tr}{Tr}
\usepackage{caption}

\newcommand{\smalltwobytwo}[4]{
\left( \begin{smallmatrix} 
  #1 & #2\\
  #3 & #4 
\end{smallmatrix}\right)}

\newcommand{\twobytwo}[4]{
\left( \begin{array}{cc} 
  #1 & #2\\
  #3 & #4 
\end{array} \right)}

\usepackage{xcolor}
\hypersetup{
    colorlinks,
    linkcolor={red!50!black},
    citecolor={blue!50!black},
    urlcolor={blue!80!black}
}

\newtheoremstyle{def}
     {10pt}
     {10pt}
     {}
     {}
     {\rmfamily\bfseries\upshape}
     {.}
     {.5em}
     {}
 \theoremstyle{def}
\newtheorem{rem}[thm]{Remark}
\newtheorem{defn}[thm]{Definition}

\begin{document}

\title[Indecomposable vector-valued modular forms]{Indecomposable vector-valued modular forms and periods of modular curves}

\author[L. Candelori]{Luca Candelori}
\author[T. Hartland]{Tucker Hartland}
\author[C. Marks]{Christopher Marks}
\author[D. Y\'{e}pez]{Diego Y\'{e}pez}

\address{Department of Mathematics, University of Hawaii, 2565 McCarthy Mall, Honolulu, HI 96822, USA}
\email{candelori@math.hawaii.edu}

\address{Department of Applied Mathematics, University of California, Merced, 5200 N Lake Road, Merced, CA, 95343}
\email{thartland@ucmerced.edu}

\address{Department of Mathematics and Statistics, California State University, Chico, 400 West First Street, Chico, CA 95929, USA}
\email{cmarks@csuchico.edu}

\address{}
\email{dyepez19@yahoo.com }

\keywords{Indecomposable representations, modular forms, periods}

\begin{abstract} We classify the three-dimensional representations of the modular group that are reducible but indecomposable, and their associated spaces of holomorphic vector-valued modular forms. We then demonstrate how such representations may be employed to compute periods of modular curves. This technique obviates the use of Hecke operators, and therefore provides a method for studying noncongruence modular curves as well as congruence.
\end{abstract}

\maketitle

\section{Introduction}

For a long while, modular forms have been an indispensable tool in the theory of numbers. Perhaps in part because Frobenius was so separated in time from Jacobi, Eisenstein, and the other early adopters of modularity, the use of representation theory of the modular group as a means for studying modular forms has been a comparatively recent development. In a certain sense, this point of view is entirely natural since if $G$ is a normal subgroup of the modular group $\G=\PSL_2(\bb{Z})$ and $k$ is any even integer, then $\G$ acts on the space of weight $k$ modular forms for $G$ (either meromorphic or holomorphic) via the usual ``slash'' action
$$f\mapsto f|_k\gamma(\tau)=f\left(\frac{a\tau+b}{c\tau+d}\right)(c\tau+d)^{-k}$$
in weight $k$ that defines (weak) modularity on $G$. Selberg \cite{Selberg}, for example, made good use of this point of view in improving bounds on the growth of Fourier coefficients of cusp forms for arbitrary finite index subgroups of $\G$, and later on Eichler and Zagier \cite{EZ} pointed out that one may define Jacobi forms using this point of view together with Jacobi theta functions. 

More recently, the importance of this action of the full modular group on spaces of modular forms has been made clear in the growing unification of number theory and high energy physics, where e.g.\ the work of Zhu \cite{Zhu} shows that certain vector-valued modular forms have Fourier coefficients that count the dimensions of graded spaces of modules for rational vertex operator algebras (VOAs). Largely motivated by the connection with rational VOAs, Knopp and Mason \cite{KM} initiated a formal study of vector-valued modular forms, and this has led to a significant amount of ongoing research in the subject. 

One of the more novel aspects of this representation theoretic approach to modular forms is that one may study more general functions that attain modular invariance only on an infinite index subgroup of $\G$. Since the Riemann surface that arises from the action of such a subgroup on the union $\hstar$ of the complex upper half-plane $\bb{H}$ and the cusps $\bb{Q}\cup\{i\infty\}$ of $\G$ is not compact, these functions are not constrained by e.g.\ Liouville's theorem, and consequently one may observe in this situation modular functions that are holomorphic throughout all of $\hstar$. A classic example of this occurs when one fixes a base point $\tau_0\in\hstar$ and integrates a weight two cusp form $f$ on a finite index subgroup $G\leq\G$, obtaining a holomorphic function $u(\tau)=\int_{\tau_0}^\tau f(z)\,dz$ that first becomes modular on some infinite index subgroup of $G$. The periods associated with $f$ for the modular curve $G\backslash\hstar$ are then obtained by evaluating $u$ at $\sigma\tau_0$ for appropriate hyperbolic elements $\sigma\in G$. In general it is not so easy to determine explicitly these periods but, as we demonstrate in Section \ref{section:periods} below, one may obtain them by studying the action of $\G$ on the full space $S_2(G)$ of weight two cusp forms for $G$. In this way, one realizes these periods as columns of the matrices $\rho(\sigma)$, where $\rho$ is a representation of $\G$ that encodes the action of $\G$ on the integrals of the weight two cusp forms and $\sigma$ denotes the above-mentioned hyperbolic elements of $G$. This provides what appears to be a novel method for obtaining periods of modular curves, one that in particular does not require the use of Hecke operators. This allows one to obtain periods for noncongruence modular curves, i.e.\ curves $G\backslash\hstar$ with $G$ a finite index,  noncongruence subgroup of $\G$, even though the action of the Hecke algebra associated to such a subgroup is known to be defective \cite{Thompson}.

The representations of the modular group that occur in this context are indecomposable but reducible, a phenomenon that can occur only in the infinite image setting, and in order to reliably compute periods of modular curves of genus $g$ via this method one requires a classification of representations $\rho:\G\rightarrow\gln{g+1}$ of the form
\begin{equation}
\label{eq:RhoExtension}
0 \rightarrow \rho_0 \rightarrow \rho \rightarrow \rho_1 \rightarrow 0,
\end{equation}
where the $g$-dimensional subrepresentation $\rho_0$ gives the action of $\G$ on $S_2(G)$ (here the genus $g$ subgroup $G$ is assumed to be normal in $\G$ for simplicity) and the quotient $\rho_1=1$ is trivial. Thus we are motivated to study the more general problem of classifying representations \eqr{RhoExtension} together with their associated spaces $M(\rho)$ of holomorphic vector-valued modular forms. The {\em free module theorem} for vector-valued modular forms \cite{MarksMason} asserts that in this context $M(\rho)$  is free of rank $r=\dim\rho$ over the graded ring $R=M(1)$ of scalar-valued modular forms for $\G$. Thus the classification of the $M(\rho)$ comes down to determining the weights of the $r$ free generators for $M(\rho)$, the so-called {\em generating weights} for $\rho$. For representations of low dimension, it is indeed possible to list all such representations $\rho$ and their corresponding generating weights. For dimension two, this classification has already been carried out in \cite[Sec 4]{MarksMason}, so a primary aim of this article is to accomplish a similar classification for representations \eqr{RhoExtension} of dimension three. In Section \ref{section:3DimIndecomposable} below we classify all such extensions, and the results may be summarized as follows.
\begin{thm}\label{thm:unitmain}
Let $\rho:\G\rightarrow \GL_3(\CC)$ be as in \eqref{eq:RhoExtension} with $\rho_0$ two-dimensional and suppose $\rho\smalltwobytwo{1}{1}{0}{1}$ is diagonalizable. Then up to equivalence there are 
\begin{itemize}
\item[(i)] 6 $\rho$ with $\rho_0$ a direct sum of characters $\G\rightarrow\cstar$, as tabulated in Table \ref{table1} below.
\item[(ii)] 18 $\rho$ with $\rho_0$ indecomposable but reducible, as tabulated in Tables \ref{table2} and \ref{table3} below.
\item[(iii)] 2 $\rho$ for any choice of irreducible $\rho_0$ such that the ratio of the eigenvalues of $\rho_0(T)$ is different from -1 (there are infinitely many such choices).
\end{itemize}
\end{thm}
The assumption that $\rho\smalltwobytwo{1}{1}{0}{1}$ is diagonalizable simplifies some of the arguments, and for our application to the computations of periods of modular curves it suffices to classify such representations. A similar classification may be worked out for indecomposable $\rho$ with $\rho\smalltwobytwo{1}{1}{0}{1}$ not diagonalizable, but we omit this for the sake of brevity. 

By studying the dual representations of the above $\rho$, one is easily led to the following result which covers the cases when $\rho_0$ is one-dimensional:
\begin{thm}\label{thm:onedim}
Suppose that $\dim(\rho_0)=1$ in \eqref{eq:RhoExtension}. Then $\rho$ is equivalent to one of the representations in $(i)$ or $(ii)$ of Theorem \ref{thm:unitmain} if and only if $\rho$ admits a two-dimensional invariant subspace. Otherwise the quotient representation $\rho_1$ is irreducible and $\rho$ is the dual of one of the representations in $(iii)$ of Theorem \ref{thm:unitmain}.
\end{thm}

In Section \ref{section:GenWeights}, the generating weights for the associated spaces of holomorphic vector-valued modular forms are worked out. A key observation here is that, even though  $\rho$ is a non-trivial extension, it is still possible that the corresponding module of modular forms $M(\rho)$ splits as the direct sum $M(\rho_0)\oplus M(\rho_1)$. In this case we say that $\rho$ is {\em $M$-split}, and the generating weights of $\rho$ are just the union of the generating weights of $\rho_0$ and $\rho_1$. To decide when $\rho$ is $M$-split we develop sheaf-theoretic cohomological tools based on the geometric theory of vector-valued modular forms \cite{CandeloriFranc}.  It should be noted that, concurrently to the writing of this article, similar tools have been employed to analyze the structure of vector-valued modular forms for indecomposable representations of genus zero subgroups of $\Gamma$ in \cite{Genus0}. Applying these geometric methods, we obtain a list of representations $\rho$ that are $M$-split in Theorems \ref{theorem:rho0IrreducibleGenWeights}, \ref{thm:rho0indecomposablegeneratingweights}, \ref{thm:rho0CompletelyReducibleGeneratingWeights}  and \ref{thm:rho0oneDimGenWeights} below. We summarize these results as follows.
\begin{thm}
Let $\rho:\G\rightarrow \GL_3(\CC)$ be as in \eqref{eq:RhoExtension}. Then
\begin{itemize}
\item [(i)] If $\rho_0$ is two-dimensional and completely reducible, there is 1 isomorphism class of $\rho$ that is not $M$-split, and all the other classes are $M$-split. 
\item[(ii)]If $\rho_0$ is two-dimensional, indecomposable but reducible, there is 1 isomorphism class of $\rho$ that is not $M$-split, and all the other classes are $M$-split. 
\item[(iii)] For each $\rho_0$ two-dimensional, irreducible of finite image with generating weights $k_1=8, k_2=10$, there is one isomorphism class of $\rho$ that is not $M$-split. For all other $\rho_0$ irreducible of finite image, $\rho$ is $M$-split.\qed
\end{itemize}
\end{thm}
We then treat separately each of the cases where $\rho$ is not $M$-split to compute the generating weights. Note that the same cohomological tools can be applied to study the case of general $\rho_0$, not necessarily factoring through a finite group, and even to study the case of higher-dimensional representations $\rho$. The classification when $\rho_0$ one-dimensional is similarly obtained.  

In Section \ref{section:periods}, we provide the details for the above discussion regarding the use of these representations in the computation of modular curve periods (Theorem \ref{thm:MainPeriodThm}  below), and establish the following result.
\begin{thm}
Let $G\leq\Gamma$ be a normal subgroup of finite index and genus $g$. Let $\{f_1, \ldots, f_g\}$ be a basis of the space $S_2(G)$ of weight two cusp forms for $G$, and let $\rho_0$ be the representation given by the $|_2$ action of $\G$ on $S_2(G)$. Then 
$$F(\tau) =\left(\int_{i\infty}^{\tau} f_1(z)\,dz, \ldots, \int_{i\infty}^{\tau} f_g(z)\,dz,1\right)^t$$ 
is a holomorphic vector-valued modular function for an indecomposable representation $\rho$ of dimension $g+1$ of the form 
$$
0 \rightarrow \rho_0 \rightarrow \rho \rightarrow 1 \rightarrow 0,
$$
and such that $\rho\smalltwobytwo{1}{1}{0}{1}$ is diagonalizable. 
\end{thm}

As $\rho_0$ factors through the finite group $\Gamma/G$, when $g=1,2$ we can use our classification together with the classification given in \cite[Sec 4]{MarksMason} to determine $\rho$ explicitly. We demonstrate this in Section \ref{section:periods} by computing explicitly the periods for three modular curves of low genus. We also make the following observation regarding the algebraicity of periods of modular curves (Theorem \ref{thm:algebraicityOfPeriods} below):
\begin{thm}
Let $G\subseteq \G$ be a subgroup containing a finite index normal subgroup $G'\triangleleft \G$ of genus one or two. Then the period matrix of $X_G$ has entries in $\overline{\mathbb{Q}}$. 
\end{thm}

It is easy to show (see Remark \ref{rmk:noCMhighergenus} below) that the above theorem requires $G'$ to have genus one or two, and that it cannot be extended to the case of $G'$ having arbitrary genus. Our method of computing periods does however extend to groups of higher genus $g$, provided a classification of indecomposable $\G$-representations of dimension $g+1$ is given. We leave the details of these computations to further exploration. 

\subsection*{Acknowledgments}
We would like to acknowledge Bill Hoffman, Ling Long and Geoff Mason for helpful discussions. The first author would also like to thank the Mathematics Department at Chico State for the hospitality during the brief visit in which this article was initiated. We would also like to thank the referee for numerous comments and a correction to an earlier version of the manuscript. 

\subsection*{Notation}
In this article we generate the modular group $\SL_2(\bb{Z})$ by 
$$
S=\mat{0}{-1}{1}{0},\quad T=\mat{1}{1}{0}{1}
$$
subject to the relations $S^2=(ST)^3$ and $S^4=1$. We will continue to use the notation $S$ and $T$ to refer to the images of these matrices in the quotient group $\G=\PSL_2(\bb{Z})=\SL_2(\bb{Z})/\{\pm1\}$. Thus 
a function $\rho:\G\rightarrow\gln{n}$ defines a matrix representation of $\G$ iff $\rho(S)^2=\rho(ST)^3=1$. We recall here that the commutator quotient $\G/\G'$ of $\G$ is cyclic of order 6, so that any \emph{character} (i.e.\ 1-dimensional representation) of $\G$ must have order dividing 6. In particular, since $\det\rho:\G\rightarrow\cstar$ is a character of $\G$, it must be true that $\det\rho(T)=\e{\frac{x}{6}}$ for some integer $x$; here and throughout, for $r\in\bb{R}$ we set $\e{r}=e^{2\pi ir}$.

\section{Three-dimensional indecomposable representations}
\label{section:3DimIndecomposable}
In this section we classify representations $\rho:\G\rightarrow\gln{3}$ that are reducible but indecomposable. For simplicity, we assume that $\rho(T)$ is diagonalizable, and that a basis has been chosen so that  
\begin{equation}\label{eq:T3}
\rho(T)=\diag{\lambda_1,\lambda_2,\lambda_3}.
\end{equation}
Since $\rho$ is reducible, by definition it has a nontrivial invariant subspace $W\leq\bb{C}^3$, which we assume for now is two-dimensional. Then conjugation with the appropriate element of $\gln{3}$ allows us to assume that the first two standard basis vectors for $\bb{C}^3$ are a basis for $W$. This implies that 
\begin{equation}\label{eq:rhoS}
\rho(S)=\threemat{a}{b}{x}{c}{d}{y}{0}{0}{\sigma}
\end{equation}
for some complex numbers $a,b,c,d,x,y,\sigma$.  We also note that conjugation by the invertible matrix $U=\diag{u,v,w}$ leaves \eqr{T3} invariant but takes \eqr{rhoS} to 
\begin{equation}\label{eq:conj}
U\rho(S)U^{-1}=\threemat{a}{buv^{-1}}{xuw^{-1}}{cvu^{-1}}{d}{yvw^{-1}}{0}{0}{\sigma}.
\end{equation}
In particular, this shows that the property of either of $x,y$ vanishing is not affected by conjugation with a diagonal matrix as above.

Continuing with our assumption that $\rho(S)$ is as in \eqr{rhoS}, we note that the upper left $2\times2$ block of $\rho$ defines a subrepresentation 
$$\rho_0:\G\rightarrow\gln{2}$$
with 
\begin{equation}\label{eq:rho0}
\rho_0(T)=\diag{\lambda_1,\lambda_2},\ \ \rho_0(S)=\mat{a}{b}{c}{d},
\end{equation}
and the lower right entry of $\rho$ defines a character $\chi:\G\rightarrow\cstar$ with
$$\chi(T)=\lambda_3,\ \ \chi(S)=\sigma.$$
This already shows that $\lambda_3$ must be a sixth root of unity, and $\sigma = \pm 1$. We now proceed by cases, based on the nature of the subrepresentation $\rho_0$.
\subsection{$\rho_0$ is irreducible}
We first consider the case where the subrepresentation \eqr{rho0} is irreducible. This allows us to assume that $\rho(T)$ is as in \eqr{T3} and $\rho(S)$ has the form \eqr{rhoS} with $abcd\neq0$. We note that for any such $\rho_0$ we must have that $\lambda_1\neq\lambda_2$, otherwise we could conjugate $\rho_0(S)$ to a diagonal matrix without altering $\rho_0(T)$, and $\rho_0$ would be reducible (this also follows from the classification in \cite{Mason}). Using the fact that $\rho_0(S^2)=\mat{1}{0}{0}{1}$ and setting
\begin{equation}
\label{eq:ssq}
\rho(S)^2=\threemat{1}{0}{by+x(a+\sigma)}{0}{1}{cx+y(d+\sigma)}{0}{0}{1}
\end{equation}
equal to the identity matrix shows that $x=0$ if and only if $y=0$, so that $xy\neq0$ (else $\rho$ is completely reducible). Employing \eqr{conj} allows us to assume that $x=y=1$, and with this assumption we see that $b=-a-\sigma$, $c=a-\sigma$ and we have
\begin{equation}
\label{equation:formulaForS}
\rho(S)=\begin{pmatrix}a&-a-\sigma&1\\a-\sigma&-a&1\\0&0&\sigma\end{pmatrix}.
\end{equation}
Note that $a\neq\pm\sigma$ since $\rho$ is not triangularizable. Using this information, we compute and find that $\rho(ST)^3=(a_{ij})$ with
$$a_{11}=a \lambda_1(a^2 \lambda_1^2+2 \lambda_1 \lambda_2-2 a^2 \lambda_1 \lambda_2 -\lambda_2^2+a^2 \lambda_2^2)$$
$$a_{12}=-(a+\sigma)\lambda_2[a^2 \lambda_1^2+\lambda_1 \lambda_2 -2 a^2 \lambda_1 \lambda_2 + a^2 \lambda_2^2].$$

Setting $a_{12}=0$ gives $a^2=-\frac{\lambda_1\lambda_2}{(\lambda_1-\lambda_2)^2}$, and using this in the identity $a_{11}=1$ yields $a=\frac{1}{\lambda_1\lambda_2(\lambda_1-\lambda_2)}$.  Substituting our values for $a$, $a^2$, and $\sigma$, we now obtain
$$a_{13}=\frac{(\lambda_1\lambda_2+\lambda_3^2)[(\lambda_1\lambda_2)^2-\lambda_1\lambda_2^2\lambda_3+\lambda_3^4]}{\lambda_1^2\lambda_2^2\lambda_3^5},$$
$$a_{23}=\frac{(\lambda_1\lambda_2+\lambda_3^2)[(\lambda_1\lambda_2)^2-\lambda_1^2\lambda_2\lambda_3+\lambda_3^4]}{\lambda_1^2\lambda_2^2\lambda_3^5}.$$
Note that if $\lambda_1\lambda_2\neq-\lambda_3^2$ then setting $a_{13}=a_{23}=0$ forces $\lambda_1=\lambda_2$, which is false. Therefore $\lambda_3^2=-\lambda_1\lambda_2=-\det\rho_0(T)$. Either of the two possible choices of $\lambda_3$ gives a representation, and there is no restriction on the irreducible $\rho_0$ we started with. 

It remains only to determine whether a representation obtained in this way is indeed indecomposable. Write 
$$
\rho \sim \twobytwo{\rho_0}{\kappa}{0}{\chi},
$$
with $\kappa$ viewed as a function $\G \rightarrow \CC^2$. By writing down how a change-of-basis matrix $M$ can decompose $\rho$ we arrive at the following criterion:

\begin{prop}
\label{prop:decomposabilityCriterion}
The three-dimensional representation $\rho$ is decomposable as $\rho_0'\oplus\chi'$, for some two-dimensional representation $\rho_0'$ and character $\chi'$, if and only if there is a matrix 
$$M=\begin{pmatrix}A&v\\w^t&c\end{pmatrix}\in\gln{3}$$ 
with $A\in M_2(\CC)$, $v,w\in\bb{C}^2$, and $c\in\bb{C}$, such that 
\begin{align*}
A\rho_0&=\rho_0'A\\
A\kappa&=(\rho_0'-\chi)v\\
w^t\rho_0&=\chi'w^t\\
w^t\kappa&=c(\chi-\chi')
\end{align*}
\end{prop}

Specializing now to the current case where $\rho_0$ is irreducible, note that if $\rho$ were decomposable we must have $\rho_0 = \rho_0'$, since 
$$
\rho_0\cap\rho_0' \neq \{0\} 
$$
or otherwise we would have a 4-dimensional subspace $\rho_0\oplus\rho_0'$ of $\rho$. This also implies that $\chi = \chi'$ since $\det \rho = \det \rho_0\cdot \chi = \det \rho_0\cdot \chi'$. The third relation of Prop. \ref{prop:decomposabilityCriterion} shows that $w$ defines an invariant of the irreducible representation $\rho_0^*\chi^{-1}$, where $*$ denotes the dual representation. Therefore $w=0$, which implies that $A$ is invertible.

Since we are assuming that $\rho(T)$ is diagonalizable, $\kappa(T)=(0,0)^T$ implies that $\rho_0(T)v=\chi(T) v$ so either $v$ is an eigenvector of $\rho_0$ with eigenvalue $\lambda_3 = \chi(T)$, or $v=0$. We need to consider both cases:

\begin{itemize}
\item[1.] {\em $\rho(T)$ has distinct eigenvalues.} In this case $v=0$ necessarily. The second relation of Prop. \ref{prop:decomposabilityCriterion} then gives
$$
A\,\kappa(S) = A\begin{pmatrix}1\\1\end{pmatrix} =0,
$$
which is impossible since $A$ is invertible. Therefore in this case $\rho$ is indecomposable.  

\item[2.] {\em $\rho(T)$ has repeated eigenvalues.} If $\rho(T)$ has repeated eigenvalues then we either have $\lambda_1\neq\lambda_2=\lambda_3$ or $\lambda_3=\lambda_1\neq\lambda_2$. In the former case, one may check that 
$$M=\begin{pmatrix}-a-s&0&0\\0&-a-s&1\\0&0&1\end{pmatrix}$$
decomposes $\rho$ into a direct sum while while in the latter case one may use 
$$M=\begin{pmatrix}a-s&0&1\\0&a-s&0\\0&0&1\end{pmatrix}$$
to decomposes $\rho$. 

\end{itemize}

Note that only finitely many representations fall into case 2 above. Since $\lambda_3^2 = -\lambda_1\lambda_2$ always, in this case we must have 
$
\rho(T) = \diag{\lambda,-\lambda,\lambda}
$ or $
\rho(T) = \diag{-\lambda,\lambda,\lambda}
$ with $\lambda$ a sixth-root of unity.

We summarize the results of this section as follows. 

\begin{thm} 
\label{thm:Irreducible2DimSubFactor}Suppose the representation $\rho_0$ in \eqr{rho0} is irreducible, and let $\lambda_1,\lambda_2$ be the (distinct) eigenvalues of $\rho_0(T)$.
\begin{itemize}

\item[(i)] If  $\lambda_2/\lambda_1 \neq -1$, then there are two inequivalent indecomposable three-dimensional representatons $\rho\supseteq \rho_0$ such that $\rho(T)$ is diagonalizable. These are uniquely determined by the choice of $\lambda_3=\sqrt{-\lambda_1\lambda_2}$ and they are defined by $\rho(T)$, $\rho(S)$ as in \eqr{T3}, \eqr{rhoS} respectively.

\item[(ii)]If  $\lambda_2/\lambda_1 = -1$, there is no three-dimensional indecomposable $\rho\supseteq \rho_0$ with $\rho(T)$ diagonalizable.\qed
\end{itemize}
\end{thm}

This implies statement iii) of Theorem \ref{thm:unitmain}. 

\begin{rem}
If $\rho_0$ is one of the (finitely many) representations of Theorem \ref{thm:Irreducible2DimSubFactor}, part (ii), then we can still find exactly two inequivalent indecomposable three-dimensional representations $\rho\supseteq \rho_0$, but $\rho(T)$ will not be diagonalizable. 
\end{rem}

\subsection{$\rho_0$ is reducible but indecomposable}
We now consider the case where the sub-representation \eqr{rho0} is reducible but indecomposable. Thus we may assume that $c=0$ but $b\neq0$ in \eqr{rhoS}. Note that \eqr{ssq} implies that $x\neq0$, since otherwise we would have $x=y=0$ and $\rho$ would be completely reducible. It may, however, occur that $y=0$, so we assume this first. We find from \eqr{conj} that conjugation by a diagonal matrix allows us to assume that $b=x=1$, and \eqr{ssq} shows that $a=-\sigma=-d$ with $\sigma^2=1$, so we have
\begin{equation}\label{eq:S2}
\rho(S)=\threemat{-\sigma}{1}{1}{0}{\sigma}{0}{0}{0}{\sigma}.
\end{equation} 
We compute and find (again using using $\sigma^2=1$) that
$$\rho(ST)^3=\threemat{-\sigma\lambda_1^3}{\lambda_2(\lambda_1^2-\lambda_1\lambda_2+\lambda_2^2)}{\lambda_3(\lambda_1^2-\lambda_1\lambda_3+\lambda_3^2)}{0}{\sigma\lambda_2^3}{0}{0}{0}{\sigma\lambda_3^3}.$$
Setting this equal to the identity matrix shows that $\sigma=-\lambda_1^3=\lambda_2^3=\lambda_3^3$, and the ratios $\frac{\lambda_2}{\lambda_1}$ and $\frac{\lambda_3}{\lambda_1}$ are primitive cube roots of $-1$. In this case, one sees that conjugation by $U=\threemat{1}{0}{0}{0}{0}{1}{0}{1}{0}$ leaves \eqr{S2} invariant but takes \eqr{T3} to $U\rho(T)U^{-1}=\diag{\lambda_1,\lambda_3,\lambda_2}$, 
so the relabeling of $\lambda_2$ and $\lambda_3$ does not give a different representation. Writing $\lambda_j=\e{\frac{x_j}{6}}$ for integers $0\leq x_j\leq5$, we get 18 possible triples $(x_1,x_2,x_3)$.

We apply Prop. \ref{prop:decomposabilityCriterion} to each of these triples to detect whether the corresponding representations are decomposable. This computation again breaks down into two cases
\begin{itemize}
\item[1.]{\em $\rho(T)$ has distinct eigenvalues.} Let $\rho$ be a representation corresponding to a triple $(x_1,x_2,x_3)$, and suppose all $x_i$ are distinct (there are 6 such triples, listed in Table \ref{table2}). In this case $\rho_0$ is a two-dimensional indecomposable given by
$$
0 \rightarrow \chi^{x_1} \rightarrow \rho_0 \rightarrow \chi^{x_2} \rightarrow 0. 
$$ 
If 
$
\rho \simeq \rho_0'\oplus \chi'
$
then we necessarily must have that $\rho_0'$ is a 2-dim indecomposable with
$$
\chi^{x_1} \subseteq \rho_0',
$$
by dimension reasons. If $\rho_0'=\rho_0$, then $\chi = \chi'$ as well. Applying Prop. \ref{prop:decomposabilityCriterion}, we see that $w$ is an invariant of $\rho_0^*\cdot\chi^{-1}$. By inspection on all 6 triples, this is only possible if $w=0$, which implies that $A$ is invertible. Since $\kappa(T)= (0,0)^T$, we have that $v=0$ since all $\rho(T)$-eigenvalues are distinct. Therefore
 $$
A\,\kappa(S) = A\begin{pmatrix}1\\0\end{pmatrix} =0,
$$
which is impossible since $A$ is invertible. Therefore $\rho_0\neq \rho_0'$, but since $\rho_0'$ is reducible, indecomposable, containing $\chi^{x_1}$ there is only one possibility for $\rho_0'$ (\cite{MarksMason}). Arguing again as above, we can exclude this possibility as well for all 6 triples. Therefore for all these triples $\rho$ is indecomposable.

\item[2.]{\em $\rho(T)$ has repeated eigenvalues.} For the remaining triples $(x_1,x_2,x_3)$ where one of the $x_i$ is repeated, the corresponding representation $\rho$ can be explicitly decomposed by
$$
M= \threemat{1}{0}{0}{1}{1}{1}{0}{0}{1}.
$$
\end{itemize}

The 6 triples $(x_1,x_2,x_3)$ (Case 1 above) giving indecomposable representations with $y=0$ in are listed in Table \ref{table2} below.

\begin{center}
\captionof{table}{$\rho_0$ reducible, indecomposable and $y=0$}
\label{table2}
 \begin{tabular}{c c c} 
 $x_1$\ &$x_2$\ &$x_3$\\ 
 \hline
1&2&0\\5&4&0\\2&3&1\\0&5&1\\3&4&2\\4&5&3 
 \end{tabular}
\end{center}
This gives 6 of the 18 equivalence classes in statement ii) of Theorem \ref{thm:unitmain}.

The final case for upper triangular representations occurs when $c=0$ but $bxy\neq0$ in \eqr{rhoS}. In this case we can assume (using \eqr{conj}) that $x=y=1$, and setting \eqr{ssq} equal to the identity matrix shows that
$$\rho(S)=\threemat{\sigma}{-2\sigma}{1}{0}{-\sigma}{1}{0}{0}{\sigma}$$
with $\sigma^2=1$. Using this identity, we now compute and find that 
$$\rho(ST)^3=\threemat{\sigma\lambda_1^3}{-2\sigma\lambda_2(\lambda_1^2-\lambda_1\lambda_2+\lambda_2^2)}{\lambda_3(\lambda_1^2-2\lambda_1\lambda_2+2\lambda_2^2+\lambda_1\lambda_3-2\lambda_2\lambda_3+\lambda_3^2)}{0}{-\sigma\lambda_2^3}{\lambda_3(\lambda_2^2-\lambda_2\lambda_3+\lambda_3^2)}{0}{0}{\sigma\lambda_3^3}.$$
The diagonal entries show that $\sigma=\lambda_1^3=-\lambda_2^3=\lambda_3^3$, so $\frac{\lambda_2}{\lambda_1}$ and $\frac{\lambda_2}{\lambda_3}$ are each cube roots of $-1$, and the super diagonal entries show that these roots are primitive, since they each solve the equation $x^2-x+1=0$. Using the identities defined by these entries, the final nonzero entry yields the additional constraint $\lambda_1\lambda_3-\lambda_1\lambda_2-\lambda_2\lambda_3=0$, which in terms of primitive roots is the relationship $\frac{\lambda_2}{\lambda_1}+\frac{\lambda_2}{\lambda_3}=1$. This shows either $\frac{\lambda_2}{\lambda_1}=\e{\frac{1}{6}}$ and $\frac{\lambda_2}{\lambda_3}=\e{\frac{5}{6}}$, or vice versa. Writing $\lambda_j=\e{\frac{x_j}{6}}$ with $0\leq x_j\leq5$, we obtain 12 representations corresponding to triples $(x_1,x_2,x_3)$. These are listed in Table \ref{table3}.

\begin{center}
\captionof{table}{$\rho_0$ reducible, indecomposable and $y\neq0$}
\label{table3}
 \begin{tabular}{c c c} 
 $x_1$\ &$x_2$\ &$x_3$\\ 
 \hline
 5&0&1\\1&0&5\\0&1&2\\2&1&0\\1&2&3\\3&2&1\\2&3&4\\4&3&2\\3&4&5\\5&4&3\\4&5&0\\0&5&4
 \end{tabular}
\end{center}

All the triples in Table \ref{table3} give indecomposable representations, by applying the same argument as in Case 1 above. This completes the proof of statement ii) in Theorem \ref{thm:unitmain}.

\subsection{$\rho_0$ is completely reducible}
If $\rho_0$ is completely reducible, we may further assume that $b=0$ in \eqr{rhoS} as well as $c=0$. Note that in this case we must have $xy\neq0$, since otherwise  $\rho$ is the direct sum of a one- and a two-dimensional subrepresentation ($x=0,y\neq0$ or $x\neq0,y=0$) or three one-dimensional subrepresentations ($x=y=0$). Now \eqr{conj} shows that we may take $x=y=1$, and setting \eqr{ssq} equal to the identity matrix (and using $b=0$) shows that $a=d=-\sigma$. Thus
$$\rho(S)=\uptrithree{-\sigma}{0}{1}{-\sigma}{1}{\sigma}$$
and
$$\rho(ST)=\uptrithree{-\sigma\lambda_1}{0}{\lambda_3}{-\sigma\lambda_2}{\lambda_3}{\sigma\lambda_3}.$$
One then computes and finds (using $\sigma^2=1$) that 
$$\rho(ST)^3=\uptrithree{-\sigma\lambda_1^3}{0}{\lambda_3(\lambda_1^2+\lambda_3^2-\lambda_1\lambda_3)}{-\sigma\lambda_2^3}{\lambda_3(\lambda_2^2+\lambda_3^2-\lambda_2\lambda_3)}{\sigma\lambda_3^3}.$$
Setting this equal to the identity matrix, we see that
\begin{equation}\label{eq:6throot1} 
-\lambda_1^3=-\lambda_2^3=\lambda_3^3=\sigma,
\end{equation}
so each of the $\lambda_j$ are sixth roots of unity. The off-diagonal entries show that the ratios $\frac{\lambda_3}{\lambda_1}$ and $\frac{\lambda_3}{\lambda_2}$ are primitive cube roots of $-1$, since each is a solution of $x^2-x+1=0$. This is also implied by \eqr{6throot1}, so any sixth roots of unity $\lambda_j$ satisfying \eqr{6throot1} give a representation in this setting and, furthermore, the $\lambda_j$ define all the entries of $\rho(S)$ explicitly. Note that the choice of ordering of $\lambda_1$ and $\lambda_2$ is irrelevant. With this in mind, and writing $\lambda_j=\e{\frac{x_j}{6}}$ for integers $0\leq x_j\leq5$, we obtain eighteen possible triples $(x_1,x_2,x_3)$. 

When one of the $x_i$'s is repeated, the matrix 
$$
M= \threemat{0}{-1}{1}{0}{1}{0}{-1}{1}{0}
$$
decomposes $\rho$. When the $x_i$'s are distinct, we can again apply Prop. \ref{prop:decomposabilityCriterion} as before to conclude that $\rho$ is indecomposable. We are thus left with 6 triples $(x_1,x_2,x_3)$ such that $\rho$ is irreducible, listed in Table \ref{table1} below.
\begin{center}
\captionof{table}{$\rho_0$ completely reducible}
\label{table1}
 \begin{tabular}{c c c} 
 $x_1$\ &$x_2$\ &$x_3$\\ 
 \hline
1&5&0\\
2&0&1\\
3&1&2\\
4&2&3\\
5&3&4\\
0&4&5\\
 \end{tabular}
\end{center}
This establishes part i) of Theorem \ref{thm:unitmain}.

\subsection{$\rho_0$ is one-dimensional}
We now wish to consider the case where $\rho$ has a one-dimensional invariant subspace but no two-dimensional invariant subspace. Note that such a representation may be obtained by taking the dual of one of the representations classified above, and a $\rho$ of the above form has dual $\rho^\ast$ (where $\rho^\ast(\gamma)=\rho(\gamma^{-1})^t$ for each $\gamma\in\G$) with
$$\rho^\ast(T)=\diag{\lambda_1^{-1},\lambda_2^{-1},\lambda_3^{-1}},\ \ \rho^\ast(S)=\threemat{a}{c}{0}{b}{d}{0}{x}{y}{\sigma}.$$
Conjugation by $\threemat{0}{0}{1}{0}{1}{0}{1}{0}{0}$ allows us to assume that the first standard basis vector spans the resulting one-dimensional invariant subspace, and this gives  the representation
\begin{equation}\label{eq:dual}
\rho'(T)=\diag{\lambda_3^{-1},\lambda_2^{-1},\lambda_1^{-1}},\ \ \rho'(S)=\threemat{\sigma}{y}{x}{0}{d}{b}{0}{c}{a}.
\end{equation}
Thus the classification will be complete once we determine which of the above $\rho'$ are not conjugate to any $\rho$ as in \eqr{rhoS}, \eqr{T3}. We may assume via conjugation that $c=0$ if and only if \eqr{rho0} is reducible, and in this case it is evident that $\rho'$ is one of the representations already classified.  This establishes the first statement in Theorem \ref{thm:onedim}. On the other hand, if \eqr{rho0} is irreducible then $\rho$ cannot have a one-dimensional invariant subspace (or $\rho$ would be completely reducible), so the dual $\rho'$ is not one of the representations we have already classified. This implies the second statement of Theorem \ref{thm:onedim}.

\section{Generating Weights and the M-functor}
\label{section:GenWeights}

Let $\rho: \SL_2(\ZZ)\rightarrow \GL(V)$ be an indecomposable, finite-dimensional complex representation.  For any $k\in \ZZ$, let $M_k(\rho)$ be the (finite-dimensional) complex vector space of {\em holomorphic} $\rho$-valued modular forms (\cite{CandeloriFranc},\cite{MarksMason}) of weight $k$, and let 
$$
M(\rho) := \bigoplus_{k\in\ZZ} M_k(\rho)
$$ 
be the $\ZZ$-graded module of $\rho$-valued modular forms over the ring $R = M(1) = \CC[E_4,E_6]$ of modular forms of level one. By \cite[Thm 1]{MarksMason} this module is free of rank $d = \dim \rho$, thus any choice of homogeneous generators gives an isomorphism
$$
M(\rho) \simeq \bigoplus_{i=1}^{d = \dim \rho} R[-k_i],
$$
where by $R[a]$ we denote the rank one graded module over $R$ obtained by shifting the grading by $a$. The $d$-tuple of integers $(k_1, \ldots, k_d)$ does not depend on the choice of generators, and the $k_i$'s are called the {\em generating weights} of $\rho$. The goal of this section is to elucidate the relationship between the generating weights of indecomposable, reducible representations $0\rightarrow \rho_0 \rightarrow \rho \rightarrow \rho_1 \rightarrow 0$ and the generating weights of $\rho_0,\rho_1$. For simplicity we almost always assume that the generating weights of $\rho_0$ and $\rho_1$ all lie in $\{0,\ldots, 11\}$ (this is always the case for unitarizable, or even {\em positive} representations \cite{CandeloriFranc}, \S 6), but the same methods apply in full generality.

Let $\mathrm{Rep}(\SL_2(\ZZ))$ be the category of finite-dimensional complex representations of $\SL_2(\ZZ)$. In \cite{MarksMason}, it is shown that the functor
\begin{align*}
M:\mathrm{Rep}(\SL_2(\ZZ))&\longrightarrow \mathrm{grMod}_R \\
\rho &\longmapsto M(\rho)
\end{align*}
to the category of finitely generated graded $R$-modules is faithful and left exact, but not right exact. In particular, given an exact sequence of representations
\begin{equation}
\label{eqn:generalRepsExactSequence}
0\rightarrow \rho_0 \rightarrow \rho \rightarrow \rho_1 \rightarrow 0,
\end{equation}
applying the $M$-functor we obtain a long exact sequence of graded vector spaces 
$$
0 \rightarrow M(\rho_0) \rightarrow M(\rho) \rightarrow M(\rho_1) \stackrel{\delta}\rightarrow K,
$$
where $K$ is, in general, non-zero. To obtain control over this `error term', recall the construction of \cite{CandeloriFranc}, which assigns to each $\rho \in \mathrm{Rep}(\SL_2(\ZZ))$ a vector bundle $\cV(\rho)$ over the compact complex orbifold $\cMbar:=\SL_2(\ZZ)\backslash\mathfrak{h}\cup\{\infty\}$, which is just the {\em canonical extension} to $\cMbar$ of the local system over $\SL_2(\ZZ)\backslash\mathfrak{h}$ determined by $\rho$. The relationship between $M(\rho)$ and $\cV(\rho)$ can be given as follows. Let $\cL_k$ be the line bundle of holomorphic modular forms on $\cMbar$, characterized by 
$$
H^0(\cMbar, \cL_k) = M_k(1) = \text{holomorphic modular forms of weight $k$}.
$$
Then $\cV(\rho)$ is a vector bundle whose defining property is that the global sections of $\cV_k(\rho):=\cV(\rho)\otimes\cL_k$ are precisely the holomorphic $\rho$-valued modular forms of weight $k$, i.e.
$$
H^0(\cMbar, \cV_k(\rho)) = M_k(\rho) = \text{holomorphic $\rho$-valued modular forms of weight $k$.}
$$
Using the basic properties of sheaf cohomology, we may thus express $K$ as
$$
K = \bigoplus_{k \in \ZZ} H^1(\cMbar, \cV_k(\rho_0)).
$$
By standard vanishing theorems in algebraic geometry, $K$ is a finite-dimensional (graded) vector space. The linear map $\delta= \oplus_k \delta_k$ is also graded, and by the above it is non-zero only for finitely many $k\in \ZZ$. 
\begin{defn}
The representation $\rho$ is said to be {\em $M$-split} if $\delta = 0$. 
\end{defn}

Note that if $\rho$ is $M$-split then there is an $R$-module isomorphism 
$$
M(\rho) \simeq M(\rho_0) \oplus M(\rho_1),
$$
and therefore the generating weights of $\rho$ are just the union of the generating weights of $\rho_0$ and $\rho_1$. We thus determine a criterion for when $\rho$ is $M$-split:

\begin{thm}
\label{theorem:keyMSplitProp}
Suppose the generating weights of $\rho_0,\rho_1$ lie in $[0,11]\cap\ZZ$. Then $\rho$ is $M$-split unless at least one of the following occurs:
\begin{itemize}
\item[(i)] $k=0$ is a generating weight of $\rho_1$ and  $k=10$ is a generating weight of $\rho_0$
\item[(ii)]$k=1$ is a generating weight of $\rho_1$ and  $k=11$ is a generating weight of $\rho_0$
\end{itemize}
\end{thm}

\begin{proof}
Suppose $k_1, \ldots, k_n$ are the generating weights of $\rho_0$, Then $K = K_0\oplus K_1$ is only graded in degrees $k=0,1$, since 
$$
K = \bigoplus_{k\geq 0} H^1(\cMbar, \oplus_{i=1}^n \cL_{k-k_i}),
$$
and $H^1(\cMbar, \cL_k ) = 0$ for all $k=0,-1,\ldots,-9, -11$ and it is a one-dimensional vector space for $k=-10$. More precisely, if we denote by $m_0, \ldots, m_{11}$ the multiplicity of each weight of $\rho_0$, we have
$$
K_0 = H^1(\cMbar, \cL_{-10}^{\oplus m_{10}} ),\quad K_1 = H^1(\cMbar, \cL_{-10}^{\oplus m_{11}} ).
$$

 Therefore if $m_{10} = 0$ and $m_{11} = 0$ then $K=0$ and $\rho$ is $M$-split. Suppose then that $m_{10}\neq 0$. Applying the $M$-functor we get an exact sequence of $R$-modules
$$
0 \rightarrow M(\rho_0) \rightarrow M(\rho) \rightarrow M(\rho_1) \stackrel{\delta}\rightarrow K_0
$$
so that $\delta= \delta_0$ can only be non-zero in degree 0, where it is given by
$$
M_0(\rho_1) \stackrel{\delta_0}\rightarrow K_0.
$$
But if $k=0$ is not a generating weight of $\rho_1$, then $M_0(\rho_1) = 0$ and thus $\delta = 0$, i.e. $\rho$ is $M$-split. The case when $m_{11}\neq 0$ is obtained similarly by looking at degree 1.
\end{proof}

\begin{rem}
The hypotheses of Theorem \ref{theorem:keyMSplitProp} always hold whenever $\rho_0, \rho_1$  are unitarizable, or even {\em positive} (as in \cite{CandeloriFranc}, \S 6).
\end{rem}

\begin{rem}
Note that Theorem \ref{theorem:keyMSplitProp} can be viewed as a generalization to higher dimensions of \cite{MarksMason}, Theorem 4, which covers the case of $\rho$ being 2-dimensional. The result of \cite{MarksMason} is obtained by different means using the modular derivative $D$.  
\end{rem}

\subsection{2-dimensional indecomposable representations}
\label{section:dimension1and2}

We now employ Theorem \ref{theorem:keyMSplitProp} to compute the generating weights for reducible, indecomposable two-dimensional representations of $\SL_2(\ZZ)$. As is well-known, there are exactly 12 one-dimensional representations $\chi^a$, $a=0,\ldots, 11$ of $\SL_2(\ZZ)$, characterized by 
$$
\chi^a(T) = \e{\frac{a}{12}}. 
$$
As shown in \cite{CandeloriFranc}, in this case we have
$$
\cV(\chi^a) \simeq \cL_{-a}
$$
and therefore the unique generating weight of $\chi^a$ is just the integer $a \in [0,11]$. 

Next, suppose that we have a two-dimensional representation $\rho$ which is an extension 
\begin{equation}
\label{eqn:2DimExactSequence}
1\rightarrow \chi^a \rightarrow \rho \rightarrow \chi^b \rightarrow 1,
\end{equation}
and suppose that $\rho$ is indecomposable. By \cite[Lem 4.3]{MarksMason} there are precisely 24 isomorphism classes of such representations. These isomorphism classes are in bijection with ordered pairs $(a,b) \in \{0,\ldots, 11\}^2$ such that $a-b \equiv \pm 2 \mod 12$. A representative $\rho_{(a,b)}$ in each isomorphism class can be chosen to have
$$
\rho_{(a,b)}(T)  = \smalltwobytwo{\e{\frac{a}{12}}}{0}{0}{\e{\frac{b}{12}}},$$ 
so that the pair $(a,b)$ corresponds to the presentation of $\rho$ as in \eqref{eqn:2DimExactSequence}.
\begin{thm}
\label{theorem:2Dim}
\noindent

\begin{itemize}
\item[(1)] If $(a,b)\neq (10,0),(11,1)$, then the generating weights of $\rho_{(a,b)}$ are $(k_1,k_2) = (a,b)$. 
\item[(2)] If $(a,b) = (10,0), (11,1)$ the generating weights of $\rho_{(a,b)}$ are $(4,6)$, $(5,7)$, respectively. 
\end{itemize}
\end{thm}

\begin{proof}
Part (1) follows from Theorem \ref{theorem:keyMSplitProp}, since in this case $\rho$ is $M$-split, and the generating weights of $\rho$ are just $(a,b)$. In part (2), we do not know a priori whether $\rho$ is $M$-spit. We want to show that in fact $\rho$ is not $M$-split. Suppose first $(a,b) = (10,0)$. By applying the global section functor we obtain an exact sequence of vector spaces
$$
0 \rightarrow H^0(\cMbar, \cV(\rho) )  \rightarrow H^0(\cMbar,\cV(1)) \stackrel{\delta}\rightarrow H^1(\cMbar,\cV(\chi^{10}))\rightarrow H^1(\cMbar, \cV(\rho) ) \rightarrow 0,
$$
where both terms $H^0(\cMbar,\cV(1)), H^1(\cMbar,\cV(\chi^{10}))$ are one-dimensional since $\cV(1) = \cL_0$ and $\cV(\chi^{10}) = \cL_{-10}$. If $\delta = 0$ (i.e. $\rho$ is $M$-split), then the generating weights of $\rho$ are $\{0,10\}$. In particular, there is a modular form $F\in M_0(\rho)$ of minimal weight 0. This modular form must satisfy $F' = 0$, since $F' \in M_2(\rho) = 0$. In other words, $F\in V$ is constant, and it gives a $\rho$-invariant vector. Now since $\rho$ contains the one-dimensional subrepresentation $\chi^{10}$ we may find another vector $v \in V$ such that
$$
\rho(\gamma)v = \chi^{10}(\gamma)v, \gamma\in \SL_2(\ZZ).
$$
Clearly $v$ and $F$ are linearly independent, since they are both eigenvectors of $\rho(T)$ with different eigenvalues, thus they form a basis for $\rho$. But this is impossible, since then $\rho \sim \chi^{10}\oplus 1$, which would contradict the fact that $\rho$ is indecomposable. Therefore $\delta$ is an isomorphism, so that the minimal weight is $k_1 = 4$ and the generating weights are $(4,6)$. 
For the case $(a,b) = (11,1)$, note that multiplication by $\eta^2$ gives an isomorphism of graded $R$-modules $M(\rho_{(10,0)}) \simeq M(\rho_{(11,1)})$ shifting the degree by 1. But we just proved that $M(\rho_{(10,0)})\simeq R[-4]\oplus R[-6]$ and therefore $M(\rho_{(11,1)})\simeq R[-5]\oplus R[-7]$, i.e. the generating weights of $\rho_{(11,1)}$ are $(5,7)$. 
\end{proof}

\begin{rem}
Note that the contents of Theorem \ref{theorem:2Dim} can be essentially extracted from the proof of \cite{MarksMason}, Theorem 4, where they are proved by different means using the modular derivative $D$.   
\end{rem}

\section{Generating weights in dimension three}

We now turn to computing the generating weights of indecomposable, reducible representation of dimension three. For simplicity, we restrict here to representations that factor through $\Gamma=\PSL_2(\ZZ)$ and such that $\rho_0,\rho_1$ have generating weights lying in $\{0,\ldots, 11\}$, but the same methods can be applied in full generality. First, we consider those representations $\rho$ that have a two-dimensional subrepresentation, i.e. they are of the form 
\begin{equation}
\label{eqn:3DimIndec}
0 \rightarrow \rho_0 \rightarrow \rho \rightarrow \chi^{2a} \rightarrow 0,
\end{equation}
with $\rho_0$ a 2-dimensional representation and $a = 0,\ldots, 5$. Then Theorem \ref{theorem:keyMSplitProp} in this case reduces to

\begin{cor}
\label{corollary:3DimIrreducibleCriterion}
Suppose $\rho_0$ has generating weights given by $(k_1,k_2)\in \{0,\ldots, 11\}^2$. Then $\rho$ is $M$-split unless $a=0$ and either $k_1$ or $k_2$ is equal to $10$. 
\end{cor}

We use this criterion to find the generating weights of $\rho$. We divide the computations in three cases, according to whether $\rho_0$ is irreducible, reducible but indecomposable, or completely reducible.

\subsection{$\rho_0$ is irreducible}
\label{subsection:2dimIsIrreducible} According to Theorem \ref{thm:Irreducible2DimSubFactor}, for each isomorphism class of a two-dimensional, irreducible $\rho_0$ such that the ratio of the eigenvalues of $\rho_0(T)$ is different from -1, there are precisely two isomorphism classes of indecomposable representations $\rho\supseteq \rho_0$ with $\rho(T)$ diagonalizable, characterized by 
$$
\lambda_3 = \e{\frac{a}{6}} = \pm\sqrt{-\det\rho_0(T)}.
$$
For such representations, we have:
\begin{thm}
\label{theorem:rho0IrreducibleGenWeights}
Suppose $\rho_0$ is irreducible, with generating weights equal to $(k_1,k_2)\in \{0,\ldots, 11\}^2$. Then
\begin{itemize}
\item[(i)] if $(k_1, k_2) \neq (8,10)$, or $(k_1,k_2) = (8,10)$ and $a=3$, then $\rho$ is $M$-split.  
\item[(ii)] if $(k_1,k_2) = (8,10)$ and $a=0$, then $\rho$ is not $M$-split and the generating weights are $(4,6,8)$.  
\end{itemize}
\end{thm}

\begin{proof}
Let $(k_1,k_2)$ be the generating weights of $\rho_0$. Then 
$$
-\det \rho_0(T) = \e{\frac{k_1 + k_2 + 6}{12}}
$$
and thus the two choices for $\lambda_3$ (hence for $a$) are completely determined by the generating weights of $\rho_0$. Since $\rho$ has finite image, its generating weights lie in the interval $[0,11]$. The weights must be even, since all $\G$-representations have even weights, and since $\rho$ is irreducible we must have that $k_2 = k_1 + 2$. This leaves only 5 possibilities for the pair $(k_1,k_2)$. For each possible pair, we also compute in the table below both choices of $a$ such that $\rho$ is indecomposable: 

\begin{center}
\begin{tabular}{c|c}

$(k_1,k_2)$ & $2a$ \\ 
\hline 
(0,2) & 4,10 \\ 
(2,4) & 0,6 \\
(4,6) & 2,8 \\ 
(6,8) & 4,10 \\
(8,10)& 0,6 \\
\hline 
\end{tabular} 
\end{center}

By part (i) of Corollary \ref{corollary:3DimIrreducibleCriterion}, all $\rho_0$ with pairs of generating weights different than $(8,10)$ give rise to $M$-split representations $\rho$. For the pair $(8,10)$, there are two choices of $a$ given by $a=0,3$. When $a=3$, we again conclude by part (i) of  Corollary \ref{corollary:3DimIrreducibleCriterion} that $\rho$ is $M$-split. It remains to consider the case of $(k_1,k_2) = (8,10)$ and $a=0$. In this case, suppose that $\rho$ is $M$-split, so that its generating weights are $(0,8,10)$. Let $F \in M_0(\rho)$ be a modular form of minimal weight. Then $F'\in M_2(\rho)$ must be equal to 0, since there are no modular forms of weight 2 for $\rho$, thus $F \in V$ is a constant $\rho$-invariant vector. Choose a basis $v_1,v_2$ for the subrepresentation $\rho_0 \subseteq \rho$. Then $\{v_1,v_2,F\}$ must be linearly independent, since they are all eigenvectors of $\rho(T)$ corresponding to different eigenvalues. Indeed, the eigenvalue $\lambda = 1$ corresponding to $F$ cannot be an eigenvalue of $\rho_0(T)$: if we let $L$ be the unique matrix such that $e^{2\pi i L} = \rho_0(T)$ and such that the real parts of the eigenvalues of $L$ lie in $[0,1)$ (the {\em standard choice of exponents}) then we know that $L$ must satisfy $12\Tr(L) = k_1 + k_2 = 18$ (\cite{CandeloriFranc}). But if one of the eigenvalues of $\rho_0(T)$ is 1, then one of the eigenvalues of $L$ is 0, which means the other one must be equal to 18/12, contradicting the choice of logarithm. Therefore $\{v_1,v_2,F\}$ must be linearly independent, hence they must form a basis for $V$, which is impossible since then $\rho \sim \rho_0\oplus 1$. Thus in this case $\rho$ is not $M$-split. 
\end{proof}

\subsection{$\rho_0$ is reducible but indecomposable}

In this case, the $\lambda_j$ are all 6-th roots of unity and we may write $\lambda_j=\e{\frac{x_j}{6}}$, for some $x_j = 0,\ldots, 5$. The isomorphism class of $\rho$ is entirely determined by the triple $(x_1,x_2,x_3)$ and there are 18 such possible triples, listed in Tables \ref{table2} and \ref{table3}.  

\begin{thm}
\label{thm:rho0indecomposablegeneratingweights}
Suppose $\rho_0$ is reducible but indecomposable, and write $\rho_{(x_1,x_2,x_3)}$ for the unique (up to equivalence) representation determined by one of the triples $(x_1,x_2,x_3)$ of Tables \ref{table2} and \ref{table3}.
\begin{itemize}
\item[(i)] If $(x_1,x_2,x_3) \neq (5,4,0),(5,0,1),(4,5,0)$, then $\rho_{(x_1,x_2,x_3)}$ is $M$-split and its generating weights are $(2x_1,2x_2,2x_3)$. 
\item[(ii)] If $(x_1,x_2,x_3) = (5,0,1)$, then $\rho_{(x_1,x_2,x_3)}$ is also $M$-split and its generating weights are $(2,4,6)$.  
\item[(iii)] If $(x_1,x_2,x_3) = (5,4,0), (4,5,0)$, then $\rho_{(x_1,x_2,x_3)}$ is not $M$-split and its generating weights are $(4,6,8)$. 
\end{itemize}
\end{thm}

\begin{proof}
By Corollary \ref{corollary:3DimIrreducibleCriterion}, the only possible $\rho_{(x_1,x_2,x_3)}$'s that are not $M$-split must have $x_3 = 0$ and one of the generating weights of $\rho_0$ must be equal to 10. By Theorem \ref{theorem:2Dim}, we know that the generating weights of $\rho_0$ are $2x_1, 2x_2$, unless $x_1 = 5$ and $x_2 = 0$, in which case the generating weights of $\rho_0$ are $(4,6)$. By going through Tables \ref{table2} and \ref{table3} we get part (i) and (ii). For part (iii), we use the same argument as in Thm. \ref{theorem:rho0IrreducibleGenWeights}, part (ii). Namely, if $\rho$ were $M$-split then the generating weights would have to be $(0,8,10)$, which would imply that there exists an invariant vector $F\in V^{\rho}$. But $\rho_0$ is upper-triangular with $\chi^8, \chi^{10}$ on the diagonal, so $F$ is linearly independent from $\rho_0$ and we could form a basis $\{v_1,v_2,F\}$ giving a splitting $\rho \sim \rho_0\oplus 1$, which is impossible. Therefore both representations $\rho_{(5,4,0)}$ and $\rho_{(4,5,0)}$ are not $M$-split, and their generating weights can be computed by noting that in each case $M_0(\rho)=0$ and $M_k(\rho) = M_k(\rho_0)\oplus M_k(1)$ for all $k\geq 2$. 
\end{proof}

\subsection{$\rho_0$ is completely reducible} In this case $\rho_{(x_1,x_2,x_3)}$ is determined (up to equivalence) by one of the 6 triples of Table \ref{table1}.
 
\begin{thm}
\label{thm:rho0CompletelyReducibleGeneratingWeights}
Suppose $\rho_0 \sim \chi^{2x_1}\oplus\chi^{2x_2}$ is completely reducible, and write $\rho_{(x_1,x_2,x_3)}$ for the unique (up to equivalence) representation determined by one of the triples $(x_1,x_2,x_3)$ of Table \ref{table1}. 
\begin{itemize}
\item[(i)] If $(x_1,x_2,x_3) \neq (1,5,0)$, then $\rho_{(x_1,x_2,x_3)}$ is $M$-split and its generating weights are $(2x_1,2x_2,2x_3)$. 
\item[(ii)]If $(x_1,x_2,x_3) = (1,5,0)$, then $\rho_{(1,5,0)}$ is not $M$-split and its generating weights are $(2,4,6)$. 
\end{itemize}
\end{thm}

\begin{proof}
\noindent

\begin{itemize}
\item[(i)] The generating weights of $\rho_0$ are simply $(2x_1,2x_2)$, thus part (i) follows directly from Corollary \ref{corollary:3DimIrreducibleCriterion} by going through Table \ref{table1}.

\item[(ii)] For $\rho=\rho_{(1,5,0)}$, there are only two possibilities for the generating weights: $(2,4,6)$ (not $M$-split case) and $(0,2,10)$ ($M$-split case). Suppose that the latter holds. Then $M_0(\rho)$ is 1-dimensional, say generated by a modular form $F_0$ of weight 0, and $M_2(\rho)$ is also 1-dimensional, say generated by a modular form $F_2$ of weight 2. As in the proof of Theorem \ref{theorem:rho0IrreducibleGenWeights} (ii), this $F_0$ cannot be constant, for otherwise it would give an invariant vector and we could decompose $\rho$ as $1\oplus \chi^2 \oplus \chi^{10}$. For each $k\in\ZZ$, consider now the modular derivative $D$ that acts as the weight-two operator
$$
D = \frac{1}{2\pi i }\frac{d}{d\tau} - \frac{k}{12}E_2
$$  
on $M_k(\rho)$, sending an $f$ in this space to $Df\in M_{k+2}(\rho)$; here $E_2$ denotes the Eisenstein series in weight two. Applying this operator componentwise, we obtain 
$$
D_0F_0 = \lambda F_2
$$
for some $\lambda \in \CC$. On the other hand, up to a costant $F_2$ must be the modular form $(\eta^4, 0, 0)$, where $\eta$ is Dedekind's eta function, so we can assume $F_2 = \eta^4$. It follows that $D_2F_2 = 0$, that is 
$$
D^2 F_0 = 0. 
$$ 
As shown in \cite[Ex 22]{FrancMason2}, the solution space of the differential operator $D^2$ is spanned by vector-valued modular forms with respect to a 2-dimensional, indecomposable-but-reducible subrepresentation $\rho'$ of $\rho$ containing either the character $1$ or $\chi^2$ and having either $\chi^2$ or $1$ as a quotient, respectively. As explained above, $\rho'$ may not contain the character $1$ and therefore $\rho'$ is of the form
$$
0\rightarrow \chi^2 \rightarrow \rho' \rightarrow 1 \rightarrow 0. 
$$
Pick now a basis $\{v_1, v_2\}$ for $\rho'$. We also know that $\rho$ contains the character $\chi^{10}$, so let $v$ be a generator for the line defined by this character. Then $v$ cannot be linearly dependent on $\{v_1, v_2\}$, because it has eigenvalue $\e{\frac{10}{12}}$ under the action of $\rho(T)$, while $\rho'$ breaks up into two eigenspaces for $\rho(T)$ with eigenvalues $1$ and $e^{2\pi i 2/12}$. Thus $\{v_1, v_2, v\}$ is a basis for $\rho$, but this is impossible since then $\rho \simeq \rho' \oplus \chi^{10}$, contradicting the fact that $\rho$ is indecomposable. Therefore $\rho'$ is not $M$-split and the generating weights are $(2,4,6)$. 

\end{itemize}

\end{proof}

\subsection{$\rho_0$ is one-dimensional}
The only case to consider is that of $\rho$ being of the form 
$$
1 \rightarrow \chi^{2a} \rightarrow \rho \rightarrow \rho_1 \rightarrow 1,
$$
where $\rho_1 = \rho_0^*$ is the dual of an irreducible two-dimensional representation $\rho_0$, so that $\rho^*\supseteq \rho_0$ contains a two-dimensional irreducible sub-representation, as in Section \ref{subsection:2dimIsIrreducible}. 

\begin{thm}
\label{thm:rho0oneDimGenWeights}
Suppose $\rho$ has a one-dimensional sub-representation of the form $\chi^{2a}$ and a two-dimensional irreducible quotient $\rho_1$, with generating weights $k_1, k_2$. Then $\rho$ is $M$-split, so that its generating weights are $(2a,k_1,k_2)$. 
\end{thm} 

\begin{proof}
By Corollary \ref{corollary:3DimIrreducibleCriterion}, $\rho$ is always $M$-split if $2a \neq 10$. If $2a = 10$ then  $\rho^*$ is an indecomposable representation with a two-dimensional irreducible subrepresentation $\rho_0 = \rho_1^*$ and $\chi^{2}$ as a quotient. By Section \ref{subsection:2dimIsIrreducible}, this can only happen if the generating weights of $\rho_0$ are $(4,6)$. If $\rho_0(T)$ does not have 1 as an eigenvalue, then the generating weights of $\rho_1 = \rho_0^*$ are $(6,8)$, and thus $\rho$ is $M$-split by Corollary \ref{corollary:3DimIrreducibleCriterion}. If $\rho_0(T)$ has 1 as an eigenvalue, then 
$$
\rho_0^*(T) \sim \twobytwo{1}{0}{0}{\e{1/6}},
$$
which is impossible, since $\rho_1 = \rho_0^*$ is irreducible \cite[Thm 3.1]{Mason}. Therefore $\rho$ is always $M$-split. 
\end{proof}

\section{Periods of modular curves}
\label{section:periods}

We give in this final section a motivation from Riemann surface theory for classifying indecomposable representations of the modular group. Explicitly, we show how for appropriate representations of this type the entries of the matrices representing certain hyperbolic elements may be interpreted as the periods of a modular curve associated to the representation.

Suppose $X$ is a compact, connected Riemann surface of genus $g$, and let $\{\alpha_1, \ldots,\alpha_g\}$,  $\{\beta_1, \ldots \beta_g\}$ be a standard basis for $H^1(X,\ZZ)$, so that 
\begin{align*}
\langle \alpha_i, \alpha_j \rangle &= 0 = \langle \beta_i, \beta_j \rangle, \\
\langle \alpha_i, \beta_j \rangle &=  \delta_{ij}, \quad \quad i,j=1,\ldots, g 
\end{align*}
with respect to the intersection pairing $\langle \cdot, \cdot \rangle$ on $H^1(X,\ZZ)$ \cite[III]{Farkas-Kra}. Let $\Omega=\{\omega_1, \ldots, \omega_g\}$ be a basis of holomorphic differential 1-forms for $H^0(X,\Omega^1_X)$ and let $A(\Omega)$, $B(\Omega)$  be the $g\times g$-matrices whose $(i,j)$-th entries are given by
$$
A(\Omega)_{ij} = \int_{\alpha_i} \omega_j, \quad B(\Omega)_{ij} = \int_{\beta_i} \omega_j.
$$
\begin{defn}
The {\em period matrix} of $X$, with respect to the basis $\{\alpha_1, \ldots,\alpha_g, \beta_1, \ldots \beta_g\}$ for the homology $ H^1(X,\ZZ)$, is the $g\times g$-matrix given by
$$
P(X):= B(\Omega)A(\Omega)^{-1}.
$$
\end{defn}

Note that by the Riemann bilinear relations $A(\Omega)$ is invertible so that $P(X)$ is well-defined \cite[III]{Farkas-Kra}. It is clearly independent of the choice of basis $\Omega$. We also have that $P(X)$ is symmetric with positive-definite imaginary part \cite[III]{Farkas-Kra}, and  therefore $P(X) \in \mathbb{H}_g$, the $g$-dimensional Siegel's upper half-space. Note that a change of oriented basis $\{\alpha_1, \ldots,\alpha_g, \beta_1, \ldots \beta_g\}$ will change $P(X)$ by an element of $\mathrm{Sp}_{2g}(\ZZ)$.

Let now $G$ denote a finite index subgroup of $\G$ and let $g$ denote the genus of the modular curve $X_G=G\backslash\hstar$ associated to $G$. Fix the base point $i\infty$ for the homology group $H^1(X_G,\bb{Z})$ of closed paths relative to this base point. Then by \cite[Prop 1.4]{Manin} there are hyperbolic matrices $g_j\in G$, $1\leq j\leq2g$, such that the paths $\{i\infty,g_j(i\infty)\}$ give an oriented basis for $H^1(X_G,\bb{Z})$. The period matrix of $X_G$ can then be computed by  integrating a basis $\{f_1,\cdots,f_g\}$ for the space $S_2(G)\simeq H^0(X_G,\Omega^1_{X_G})$ of weight two cusp forms for $G$, over the above paths in $\hstar$. Denote these periods by
$$\Omega_{jk}=\int_{i\infty}^{g_j(i\infty)}f_k(z)\,dz.$$ 

Suppose $G$ is normal in $\G$. The $|_2$-action of $\G$ gives a linear representation of $\Gamma/G$ on $S_2(G)$, so that $F_0=(f_1,\cdots,f_g)^t$ is a weight two vector-valued cusp form for a representation $\rho_0:\G\rightarrow\gln{g}$ whose kernel contains $G$. By employing the modular Wronskian \cite{M1} one obtains a $g^{th}$ order \emph{modular differential equation} $L[f]=0$, in weight two, whose solution space is spanned by the $f_j$. 
For each $1\leq k\leq g$ let 
$$u_k(\tau)=\int_{i\infty}^\tau f_k(z)\,dz,$$
so that $\Omega_{jk}=u_k(g_j(i\infty))$, and let $F = (u_1, \ldots, u_g,1)^t$. The transformation properties of $F$ are well-known to experts, but we record them here for lack of a better reference:

\begin{thm}
\label{thm:MainPeriodThm}
$F$ is a holomorphic vector-valued modular function for an indecomposable representation $\rho: \Gamma\rightarrow \GL_{g+1}(\CC)$ of the form 
$$
0 \rightarrow \rho_0 \rightarrow \rho \rightarrow 1 \rightarrow 0,
$$
with $\rho(T)$ diagonalizable. 
\end{thm}

\begin{proof}
Consider the modular differential equation 
\begin{equation}\label{eq:mde}
L\theta[f]=0
\end{equation} 
where $\theta=q\frac{d}{dq}=\frac{1}{2\pi i}\frac{d}{d\tau}$. Since $\theta u_k=f_k$ for each $k$, the solution space of \eqr{mde} is spanned by the components of $F$, so by the covariance of the modular derivative $F$ is a vector-valued modular function for some representation $\rho:\G\rightarrow \GL_{g+1}(\bb{C})$. Note, however, that $\theta F=(F_0,0)^t$ is a weight two vector-valued cusp form for the same representation, so in fact the upper left $g\times g$ corner of $\rho$ is actually $\rho_0$ and the quotient $\rho/\rho_0$ is the trivial representation. $F$ is clearly holomorphic and not constant. Since $F$ also has weight zero, it must be that $[\G:\ker\rho]=\infty$. This implies that $\rho$ is indecomposable, otherwise it would factor through $G = \ker \rho_0$, a finite-index subgroup. Finally, note that the action of $T$ on $F(\tau)$ shows that $\rho(T)$ is diagonalizable, as can be seen, for example, by writing down the $q$-expansion of each $f_k$.  
\end{proof}

Now for $1\leq j\leq 2g$ we have
\begin{equation}
\label{equation:mainPeriodComputation}
\begin{aligned}
\left(\Omega_{j1}, \ldots, \Omega_{jk},1\right) &= F(g_j(i\infty)) \\
 &= \rho(g_j)F(i\infty) \\
 &= \rho(g_j)(0,\ldots,0, 1)
\end{aligned}
\end{equation}
by Theorem \ref{thm:MainPeriodThm}, and thus the period $\Omega_{jk}$ is equal to the $(k,g+1)$-entry of $\rho(g_j)$. This provides what seems to be a novel method for computing periods of modular curves. In particular, this method applies equally well to noncongruence or congruence subgroups, as it does not require the use of Hecke operators. Taking advantage of the classification of low-dimension representations of the modular group, we obtain the following result, which may be of independent interest.
\begin{thm}
\label{thm:algebraicityOfPeriods}
Let $G\subseteq \G$ be a subgroup containing a finite index normal subgroup $G'\triangleleft \G$ of genus one or two. Then the period matrix $P(X_G)$ has entries in $\overline{\mathbb{Q}}$. 
\end{thm}

\begin{proof}
Let $\rho_0: \Gamma \rightarrow \GL(S_2(G'))$ be the representation given by the $|_2$-action of $\G$ on $S_2(G')$. This representation factors through the finite group $\Gamma/G'$, and therefore we may choose a basis for $S_2(G')$ so that the entries of $\rho_0(T)$ are  algebraic. Since $\rho_0$ is one- or two-dimensional, it follows that all the entries of $\rho_0$ are algebraic (\cite{Mason}, \cite{MarksMason} for the two-dimensional case) as well. Let $\rho$ be the non-trivial extension of the trivial representation by $\rho_0$ given by Thm. \ref{thm:MainPeriodThm}. Then by the results of Section \ref{section:3DimIndecomposable} the entries of $\rho$ are also algebraic (this follows from \cite{MarksMason} if $\rho$ is two-dimensional). But by \eqref{equation:mainPeriodComputation}, the periods of $X_G$ with respect to the chosen basis of $S_2(G)\subseteq S_2(G')$ are linear combinations of entries of $\rho$, and thus they are algebraic. Therefore the entries of the period matrix of $X_{G}$, which do not depend on the choice of basis for $S_2(G)$, are always algebraic. Moreover, changing generators for $H^1(X_G,\ZZ)$ transforms $P(X_G)$ by a matrix in $\mathrm{Sp}_{2g}(\ZZ)$, which does not affect the field of definition of the entries of $P(X_G)$.
\end{proof}

\begin{cor}
\label{cor:CMtheorem}
Suppose $G\subseteq \G$ is of genus one and suppose it contains a normal subgroup $G'\triangleleft \Gamma$ of genus one or two. Then $X_{G}$ is an elliptic curve with complex multiplication. 
\end{cor}

\begin{proof}
As is well-known, $X_{G}$ is defined over $\overline{\mathbb{Q}}$. Moreover, by Theorem \ref{thm:algebraicityOfPeriods}, the period ratio of $X_{G}$ is algebraic. But an elliptic curve defined over a number field with algebraic period ratio must necessarily have complex multiplication, by Schneider's Theorem. 
\end{proof}

\begin{rem}
\label{rmk:noCMhighergenus}
Corollary \ref{cor:CMtheorem} relies on special properties of two- and three-dimensional representations of $\G$ and it cannot be extended to include normal subgroups $G'$ of arbitrary genus $g$ (that is, $\rho$ of arbitrary dimensions $g+1$). For example $\Gamma_0(11)$ is of genus one, and it contains the normal subgroup $\Gamma(11)$, which is of genus 26. It is well known that $X_0(11)$ does not have complex multiplication (we thank Bill Hoffman for pointing this example out to the authors). 
\end{rem}

The field of definition of the entries of the period matrix $P(X_G)$ in Theorem \ref{thm:algebraicityOfPeriods} is of low degree, and it can be determined exactly, as in the following examples. 

\subsection{Example: a genus one normal subgroup}

According to \cite{CumminsPauli}, the genus one normal congruence subgroup of smallest index is $G:=\Gamma'$, the commutator subgroup of $\Gamma$. This is a subgroup of index 6 and level 6. In this case the representation $\rho_0$ of $\Gamma$ on $S_2(G)$ is one-dimensional, so any non-zero $f\in S_2(G)$ must be a weight 2 modular form for a character of $\Gamma$. Since $M_2(\chi^a) = 0$ unless $a=2$, we conclude that $\rho_0 = \chi^2$ and that $f$ is a multiple of $\eta^4$. By \cite{MarksMason}, up to isomorphism there is a unique 2-dimensional representation $\rho$ which is a non-trivial extension of the trivial representation by $\chi^2$. This is given by 
$$
\rho(T) = \twobytwo{\e{1/6}}{0}{0}{1}, \quad \rho(S) = \twobytwo{-1}{1}{0}{1}.
$$
Note that by Theorem \ref{theorem:2Dim} the generating weights for $\rho$ are $k_1=0$ and $k_2=2$, which indeed implies that there is a non-constant $\rho$-valued modular function $F$ and a $\rho$-valued modular form of weight two given by $\theta F = (f,0)$.  
 
According to SAGE, $G$ can be generated by $A=\smalltwobytwo{2}{1}{1}{1} = T^2ST$ and $B=\smalltwobytwo{3}{-1}{1}{0} = T^3S$, both hyperbolic matrices. Evaluating $\rho$ at these two matrices gives 
$$
\rho(A) = \twobytwo{1}{\e{1/3}}{0}{1}, \quad \rho(B) = \twobytwo{1}{-1}{0}{1}
$$
and thus the two periods of $X_{G}$ are $\Omega_1 = \e{1/3}$ and $\Omega_2 = -1$. The period ratio, which does not depend on a choice of basis for $S_2(G)$, is given by
$$
P(X_G) = -\e{1/3} = \e{1/6} \in \mathbb{Q}(\e{1/6}).
$$
This computation may also be checked numerically. Choosing $\eta^4 \in S_2(G)$ as generator and $i$ as a base-point for the period integral we find that
$$
P(X_G) = \frac{\int_{i}^{Bi} \eta^4(z)\,dz}{ \int_{i}^{Ai} \eta^4(z)\,dz} \sim 0.500000... + 0.866025...\,i,
$$
by numerically approximating the integrals using the first 4 terms of the series expansion
$$
\eta^4(\tau) = q^{1/6} - 4q^{7/6} + 2q^{13/6} + 8q^{19/6} + \ldots 
$$

We note here that $G=\G'$ is one of an infinite family of genus one subgroups that are normal in $\G$, as described in \cite{Newman}. All but four of these are noncongruence, and one may compute the periods for all of these subgroups using the above technique. We leave this computation to a future publication.

\subsection{Example: the genus two normal subgroup}
\label{example:genus2}

There is only one genus two normal subgroup $G\triangleleft\G$, a congruence subgroup of index 48 and level 8, denoted by $8A^2$ in \cite{CumminsPauli}. In this case the representation $\rho_0$ of $\Gamma$ on $S_2(G)$ is two-dimensional. It cannot be reducible, for otherwise it would be a direct sum of characters each having level dividing 6. Moreover if $\rho_0$ factored through $\Gamma(4)$ then the cusp forms in $S_2(G)$ would give cusp forms in $S_2(\Gamma(4))$, which is impossible since $\Gamma(4)$ has genus zero. Therefore $\rho_0$ has level 8 and by the tables of \cite{Mason} we deduce that the only such representation possessing weight two modular forms is the one whose eigenvalues at $\rho_0(T)$ are $\e{1/8}$ and $\e{3/8}$, and whose generating weights are $(2,4)$. According to formula \eqref{equation:formulaForS} in our classification we must have 
$$
\rho(T) = \left(\begin{array}{ccc}
\e{1/8} & 0 & 0 \\ 
0 & \e{3/8} & 0 \\ 
0 & 0 & 1
\end{array}\right) , \quad 
\rho(S) = \left(\begin{array}{ccc}
-1/\sqrt{2} & -1 + 1/\sqrt{2}  & 1 \\ 
-1 - 1/\sqrt{2} & 1/\sqrt{2} & 1 \\ 
0 & 0 & 1
\end{array}\right).
$$ 
Note that by Theorem \ref{thm:Irreducible2DimSubFactor} the generating weights of $\rho$ are $(0,2,4)$, so there is a holomorphic non-constant $\rho$-valued modular function, as expected.   

Using SAGE, we find that $G$ can be generated by 
\begin{equation}
\label{equation:genus2Generators}
\begin{aligned}
\twobytwo{1}{8}{0}{1}, \;\twobytwo{15}{-4}{4}{-1}, \; \twobytwo{21}{-16}{4}{-3}, \;\twobytwo{19}{-24}{4}{-5}, \;  \twobytwo{25}{-44}{4}{-7}, \\
\twobytwo{29}{-80}{4}{-11},  \; \twobytwo{27}{-88}{4}{-13}, \;  \twobytwo{31}{-132}{4}{-17}, \;  \twobytwo{23}{-52}{4}{-9}. 
\end{aligned}
\end{equation}
Evaluating $\rho$ at each generator $A$ gives the following list of $\mathbb{R}$-linearly independent periods, appearing in the $(1,3)$- and $(2,3)$- entries of $\rho(A)$ (and rescaled by $c=-2$): 
$$
\Omega_1 = \left(\begin{array}{c}
1 \\ 
1
\end{array} \right), \quad 
\Omega_2 = \left(\begin{array}{c}i \\ 
-i
\end{array} \right),\quad
\Omega_3 = \left(\begin{array}{c}
\e{1/8} \\ 
\e{3/8}
\end{array} \right), \quad 
\Omega_4 = \left(\begin{array}{c}
\e{3/8}  \\ 
\e{1/8}
\end{array} \right).
$$
To compute the period matrix in Siegel's upper half-space $\mathbb{H}_2$, let  
$$
P_1 = \twobytwo{1}{i}{1}{-i}, \quad P_2 = \twobytwo{\e{1/8}}{\e{3/8}}{\e{3/8}}{\e{1/8}}.
$$
The period matrix of $X_G$ is then given by 
$$
P(X_G) = P_2\,P_1^{-1} = \twobytwo{\e{1/8}}{0}{0}{\e{3/8}} \in \mathbb{H}_2,
$$
with entries in $\mathbb{Q}(\e{1/8}) \subseteq \overline{\mathbb{Q}}$, as predicted by Theorem \ref{thm:algebraicityOfPeriods}. To `verify' the computation of the entries of $P_G$ numerically, note that we may choose 
\begin{align*}
f_1 &= \eta^4 \left(\frac{1728}{j}\right)^{-\frac{1}{24}}\ _2F_1\left(-\frac {1}{24}, \frac{7}{24}; \frac {3}{4}; \frac{1728}{j}\right)= q^{1/8} - q^{9/8} - 6 q^{17/8} + 5 q^{25/8}  + \ldots  \\
 f_2 &= \eta^4 \left(\frac{1728}{j}\right)^{\frac{1}{3}}\ _2F_1\left(\frac 13, \frac{2}{3}; \frac 32; \frac{1728}{j}\right) = q^{3/8} - 3 q^{11/8} + q^{19/8} + 2 q^{27/8} + \ldots
\end{align*}
as a basis for $S_2(G)$ (\cite{FrancMason1}). If we let 
$$
P_1' = \twobytwo{\int_i^{A_2i} f_1(z)\,dz }{\int_i^{A_3i} f_1(z)\,dz }{\int_i^{A_2i} f_2(z)\,dz}{\int_i^{A_3i} f_2(z)\,dz},\quad  P_2' = \twobytwo{\int_i^{A_4i} f_1(z)\,dz }{\int_i^{A_5i} f_1(z)\,dz }{\int_i^{A_4i} f_2(z)\,dz}{\int_i^{A_5i} f_2(z)\,dz},
$$
where $A_j$ is the $j$-th element of the list of generators \eqref{equation:genus2Generators}, then 
$$
P(X_G) = P_2'(P_1')^{-1} \sim \twobytwo{0.707107... + 0.707107...i}{0}{0}{-0.707107... + 0.707107...i}
$$
by numerically approximating the integrals using the first 150 terms of the  series expansion for $f_1,f_2$. 

\subsection{Example: a genus one non-normal subgroup}
Theorem \ref{thm:MainPeriodThm} may be applied equally well to modular curves for groups that are not normal in $\G$, as we demonstrate in this example. From \cite{CumminsPauli}, we see that there is a genus one non-normal subgroup $G_0 = 8D^1$ which contains the genus two subgroup $G=8A^2$ from the previous example as a normal subgroup. The group $G_0$ is of level 8 and index 24, and can be generated by the matrices
\begin{equation}
\label{equation:G0generators}
\twobytwo{0}{-1}{1}{4}, \;\twobytwo{8}{21}{3}{8}, \; \twobytwo{-12}{-29}{5}{12}, \;\twobytwo{-8}{-13}{5}{8}, \;  \twobytwo{4}{5}{3}{4}, 
\twobytwo{4}{-1}{1}{0}. 
\end{equation}
Now since $S_2(G_0)\subseteq S_2(G)$, we may pick a basis for $S_2(G)$ by first choosing a generator $f_1$ for the 1-dimensional space $S_2(G_0)$ and extending this choice to a basis $\{f_1,f_2\}$ for $S_2(G)$. Then for all $g_0\in G_0$ we must have
$$
\rho(g_0) \sim \left(\begin{array}{ccc}
1 & 0 & \Omega_1(g_0)\\ 
* & * & * \\ 
0 & 0 & 1
\end{array}\right),
$$
where $\rho$ is as in Example \ref{example:genus2} and $\Omega_1(g_0)$ is the period of $f_1$ corresponding to $g_0\in G_0$.

To put $\rho$ into this form, we may change basis via the matrix 
$$
A = \left(\begin{array}{ccc}
3+2\sqrt{2} & 0 & 0\\ 
1 & 1 & 0 \\ 
0 & 0 & 1
\end{array}\right).
$$
Evaluating $A^{t}\rho A^{t\,-1}$ at each of the 4 generating hyperbolic matrices for $G_0$ we get the periods 
$$
\Omega(g_0) = 2(2+\sqrt{2}), \;4i(1+\sqrt{2}),\; -2(2+\sqrt{2})
$$
which generate the lattice $\langle 2+\sqrt{2},\;2i(1+\sqrt{2}) \rangle\subseteq \CC$. The period ratio for the modular curve $X_{G_0}$ is therefore
$$
P(X_{G_0}) = \sqrt{2}i \in \mathbb{Q}(\sqrt{-2}).
$$
To check this computation numerically, note that
$$
f = (1+\sqrt{2})\,f_1 + 2\,f_2 = (1+\sqrt{2}) q^{1/8} + 2\,q^{3/8} - (1+\sqrt{2})  q^{9/8} + \ldots 
$$
with $f_1,f_2$ as in Example \ref{example:genus2}, is a generator for $S_2(G_0)$. We may approximate the period ratio using the first 150 terms of the $q$-series expansion of $f$ to get 
$$
P(X_{G_0}) = \frac{\int_i^{A_1i} f(z)\,dz}{\int_i^{A_6i} f(z)\,dz} \sim 1.41421... i,
$$
where $A_j$ denotes the $j$-th matrix in the list of generators.

\renewcommand\refname{References}
\bibliographystyle{alpha}
\bibliography{Dim3Weights}
\end{document}